\documentclass[11pt]{article}
\usepackage{latexsym}
\usepackage{amsfonts}
\usepackage{enumerate,graphicx}
\usepackage{verbatim}

\usepackage{amsmath, amssymb, amsthm}

\theoremstyle{definition}
\newtheorem*{acknowledgements}{Acknowledgements}

\numberwithin{equation}{section}

\newcommand\vanish[1]{}	

\newcommand\ourcomment[1]{ \textbf{[#1]} }
\newcommand\oc\ourcomment

\newtheorem{theorem}{Theorem}[section]
\newtheorem{conjecture}{Conjecture}[section]

\newtheorem{corollary}{Corollary}[section]

\newtheorem{problem}{Problem}[section]

\title{Cop number of graphs without long holes}
\author{Vaidy Sivaraman}

\begin{document}
\maketitle
\begin{abstract}
A hole in a graph is an induced cycle of length at least $4$. We give a simple winning strategy for $t-3$ cops to capture a robber in the game of cops and robbers played in a graph that does not contain a hole of length at least $t$. This strengthens a theorem of Joret-Kaminski-Theis, who proved that $t-2$ cops have a winning strategy in such graphs. As a consequence of our bound, we also give an inequality relating the cop number and the Dilworth number of a graph. 
\end{abstract}

\maketitle

 The game of cops and robbers is played on a finite simple graph $G$. There are $k$ cops and a single robber. Each of the cops chooses a vertex to start, and then the robber chooses a vertex. And then they alternate turns starting with the cop. In the turn of cops, each cop either stays on the vertex or moves to a neighboring vertex. In the robber's turn, he stays on the same vertex or moves to a neighboring vertex. The cops win if at any point in time, one of the cops lands on the robber. The robber wins if he can avoid being captured. The cop number of $G$, denoted $c(G)$, is the minimum number of cops needed so that the cops have a winning strategy in $G$.  The question of what makes a graph to have high cop number is not clearly understood. Some fundamental results were proved in \cite{AF} and \cite{NW}. The book by Bonato and Nowakowski \cite{BN} is a fantastic source of information on the game of cops and robbers. (A quick primer on the cop number is \cite{AB}.) \\

All graphs in this article are finite, simple, and connected. For graphs $H,G$ we say that $G$ is $H$-free if $G$ does not contain $H$ as an induced subgraph. For a vertex $v$, the closed neighborhood of $v$, denoted by $N[v]$, is the set of all neighbors of $v$, including $v$ itself. For a positive integer $t$, $P_t$ will denote the path graph on $t$ vertices. A hole in $G$ is an induced cycle of length at least $4$. A vertex set in a graph is said to be connected if the subgraph induced on it is connected. The Gy\'{a}rf\'{a}s path argument was first used by Gy\'{a}rf\'{a}s \cite{AG} to prove that $P_t$-free graphs are ``$\chi$-bounded" but the theorem (originally a conjecture of Gy\'{a}rf\'{a}s) that the class of graphs without long holes is ``$\chi$-bounded" was only recently proved and needed novel ideas in addition to the Gy\'{a}rf\'{a}s path argument (see \cite{SS}).  The Gy\'{a}rf\'{a}s path argument has become a standard proof technique in graph coloring (see \cite{SS}). The purpose of this note is to show that the structure of the Gy\'{a}rf\'{a}s path can be exploited to give a conceptually simple winning strategy for the cops in a graph with certain induced subgraphs forbidden. Joret, Kaminski, and Theis \cite{JKT} proved that $t-2$ cops can capture a robber in a graph without holes of length at least $t$. They claim that an argument similar to the their proof for $P_t$-free graphs (whose proof proceeds in rounds and uses induction, and some essential details are missing) works but no details are given. We prove a strengthening of their theorem. Our proof also has the advantage of being conceptually simple, with the robber being captured very quickly, and for readers familiar with the Gy\'{a}rf\'{a}s path argument, can be summarized compactly as ``Place the $t-3$ cops on a Gy\'{a}rf\'{a}s path, and slide the cops until the robber is cornered".

  \begin{theorem}\label{MAINTHEOREM}
  Let $G$ be a graph without holes of length at least $t$ ($t \geq 4$). Then $t-3$ cops can capture the robber. 
  \end{theorem}
  
  \begin{proof}
  Let $v_0 \in V(G)$. In the first move player C (the cop player) places all the $t-3$ cops in $v_0$. One of the cops is going to be stationary at $v_0$ and the other $t-4$ cops will move in the next turn. The robber will choose a vertex in $V(G) \setminus N[v_0]$. Let $v_1$ be a neighbor of $v_0$ that has a neighbor in the component $C_1$ of $G - N[v_0]$ containing the robber vertex. Now player C moves $t-4$ of his cops from $v_0$ to $v_1$. Now the robber moves to some vertex (or stays put), and let $v_2$ be a neighbor of $v_1$ that has a neighbor in the component $C_2$ of $G[C_1] - N[v_1]$ containing the robber vertex. Player C moves $t-5$ of his cops from $v_1$ to $v_2$. This procedure is repeated so that at the end of $t-2$ moves, we have an induced path $v_0-v_1- \cdots -v_{t-4}$, a nested sequence of connected vertex sets $C_1 \supseteq C_2 \cdots \supseteq C_{t-4}$, where $v_i \not  \in C_i$ but $v_i$ has a neighbor in $C_i$. There is a cop on each of the vertices $v_0, \cdots , v_{t-4}$ and the robber is in some vertex in $C_{t-4}$. This finishes the first phase. \\
  
Now we start the second phase. In every turn, the Gy\'{a}rf\'{a}s path constructed in the first phase is going to be extended, and every cop is going to move one position to the right along the Gy\'{a}rf\'{a}s path that is being built. So after one move the cops will be in $v_1, \cdots , v_{t-3}$, the robber will be in some vertex in $C_{t-3}$, and after the second move (in this phase) the cops will be in $v_2, \cdots , v_{t-2}$, the robber will be in some vertex in $C_{t-2}$ and so on. After $k$ moves in the second phase, the cops are in $v_{k}, \cdots , v_{k+t-4}$, and the robber is in some vertex in $C_{k+t-4}$. The crucial point is that, as we are going to show, the robber cannot escape from $C_{k+t-4}$, meaning he cannot move to a vertex not in $C_{k+t-4}$. Suppose he does. Say he moves from $r$ to a vertex $r' \not \in C_{k+t-4}$. Clearly $r'$ is non-adjacent to each of $v_{k}, \cdots , v_{k+t-4}$, for, otherwise the robber will be caught in the next move. By virtue of the special structure of the Gy\'{a}rf\'{a}s path and its relation to the rest of the graph, $r'$ is adjacent to some vertex on the Gy\'{a}rf\'{a}s path $v_0 - \cdots - v_{k+t-4}$. (This is the crucial step that makes the proof work, and explains why the sets $C_i$ are defined in the way mentioned.) Let $i$ be the largest index with $i < k$ such that $r'$ is adjacent to $v_i$. Now $v_i - v_{i+1} - \cdots - v_{k+t-4} - P - r - r' - v_i$ is a hole of length at least $t$, a contradiction, where $P$ is a shortest path from  $v_{k+t-4}$ to  $r$ with internal vertices in  $C_{k+t-4}$. Hence after every move (of the cops and the robber), the length of the Gy\'{a}rf\'{a}s path being built increases. (The robber is forced to do this as he cannot go out of the component he is currently in.) But the length of a Gy\'{a}rf\'{a}s path in $G$ cannot be more than $|V(G)|$. The robber will be caught in at most $|V(G)|$ moves. This completes the proof. 
 \end{proof}

The above argument gives a very transparent reason as to why a single cop has a winning strategy in a chordal graph (the case $t=4$). (Recall that a graph is chordal if it has no holes.) The proof presented here is similar to the one in \cite{VS} but has a second part where the cops move. This is necessary since the length of an induced path in a graph without a hole of length at least $t$ need not be bounded by a function of $t$. The path constructed in the proof is called a Gy\'{a}rf\'{a}s path, but in graph coloring, the component is chosen to be the one with the largest chromatic number rather than the component containing the robber. Analysis of the above proof inspires one to define the following. For each vertex $v$, let $z(v)$ denote the length of a longest induced path starting at $v$. And define $z(G)$ to be the minimum, taken over $v \in V(G)$, of $z(v)$. By choosing $v_0$ to be the vertex with minimum $z$ value in the proof above, we see that if $G$ is a graph without holes of length at least $t$ ($t \geq 4$) $t-3$ cops can capture the robber in at most $z(G)$ moves. This has the following simple corollary. (The bound on the number of cops was proved in \cite{JKT}, and the bound on the number of moves was proved in \cite{VS}.)
  
  \begin{corollary}
  Let $G$ be a $P_t$-free graph. Then $t-2$ cops can capture the robber  in at most $t-1$ moves. 
  \end{corollary}
  
  \begin{proof}
  Follows from the theorem and the obvious fact that since $G$ is $P_t$-free, it does not contain a hole of length at least $t+1$ and $z(G) \leq t-1$. 
  \end{proof}

The simplicity of the argument presented above immediately suggests: Can we do better if we understand more about the structure of graphs not containing a hole of length at least $t$? Here is a conjecture.
  
   \begin{conjecture}
   Let $G$ be a graph without a hole of length at least $t$ ($t \geq 6$). Then $t-4$ cops can capture the robber.
  \end{conjecture}
  
 The first case, namely that $2$ cops can capture a robber in a graph without holes of length at least $6$ looks particularly interesting.

  \begin{conjecture}
   Let $G$ be a graph without holes of length at least $6$. Then $2$ cops can capture the robber. 
  \end{conjecture}  
  
  Two weaker conjectures are still open.

   \begin{conjecture}\cite{VS}
   Let $G$ be a $P_5$-free graph. Then $2$ cops can capture the robber. 
  \end{conjecture}

   \begin{conjecture}\cite{ST}
   Let $G$ be a $2K_2$-free graph. Then $2$ cops can capture the robber.
  \end{conjecture}
  
  It is tempting to ask whether graphs without holes of odd length have bounded cop number. Unfortunately, even triangle-free members of the class, namely bipartite graphs, have unbounded cop number (see \cite{BN}). It is equally tempting (by analogy with recent developments in ``$\chi$-boundedness", see \cite{SS}) to ask what happens when we exclude a tree as induced subgraph. Does it force the cop number to be bounded? Well, it is well known that line graphs (which are claw-free) have unbounded cop number (see \cite{BN}). This immediately leads to a (possibly very hard) problem. 
  
  \begin{problem}
Characterize all sets $\mathcal{F}$ of graphs such that the class of $\mathcal{F}$-free graphs  has bounded cop number. 
  \end{problem}

Finally, we derive an inequality relating the cop number and the Dilworth number of a graph. First we need to give a formal definition of the Dilworth number, first defined in \cite{FH}.  \\

 For a vertex $v$, $N(v)$ will denote the set of neighbors of $v$ and $N[v] := \{v\} \cup N(v)$. A set $U$ of vertices in a graph $G$ is said to be {\it Dilworth} if for any distinct $u,v \in U$,  $N(u) \setminus N[v] \neq \emptyset$ and $N(v) \setminus N[u] \neq \emptyset$. The {\it Dilworth number} of a graph $G$, denoted $D(G)$, is the size of a largest Dilworth set. Graphs with Dilworth number $1$ are precisely the threshold graphs \cite{CH}. A full chapter is devoted to the Dilworth number in \cite{MP}, a comprehensive treatise on threshold graphs.
 
 \begin{theorem}
 For a graph $G$ with $D(G) \geq 3$, $c(G) \leq D(G) - 2.$
 \end{theorem}
 
 \begin{proof}
 Note that, for $ k \geq 4$, the Dilworth number of the cycle graph $C_k$ is $k$. Also, the Dilworth number is monotone under taking induced subgraphs i.e., the Dilworth number of an induced subgraph is at most the Dilworth number of the parent graph. Hence a graph $G$ cannot contain a hole of length $D(G)+1$. The desired inequality follows immediately from theorem \ref{MAINTHEOREM}. 
 \end{proof}
 
 It is very unlikely that the above inequality is tight, especially for larger values of $D(G)$. (It is tight for $C_4$.) It might be interesting to characterize graphs $G$ with $c(G) = D(G)-2$. 

\begin{acknowledgements}
The proof in this article was presented at the Barbados Graph Theory Workshop 2019 in Bellairs Research Institute in Holetown, Barbados, and the feedback from some of the participants was very useful in the preparation of this article. In particular, I would like to thank Paul Seymour, Sophie Spirkl, and Cemil Dibek. 
\end{acknowledgements}

\end{document}